\documentclass[a4paper,10pt,reqno, english]{amsart}
\usepackage[utf8]{inputenc}
\usepackage[T1]{fontenc}    
\usepackage{geometry} 
\usepackage{amsmath}          
\usepackage{amssymb}
\usepackage{amsthm}
\usepackage{paralist}
\usepackage{hyperref}
\usepackage{graphicx}
\usepackage{color}
\usepackage{blindtext}
\usepackage{mathtools}  
\usepackage{bbm}           
\usepackage{url}            
\usepackage{calc}          
\usepackage[mathscr]{eucal}
\usepackage{ae,aecompl}  
\usepackage{xspace} 
\usepackage{floatpag}      
\hypersetup{
  colorlinks = true,
  citecolor=black,
  linkcolor= black,
  urlcolor = black,
  pdftitle = {Semi-algebraic sets of f-vectors}
  pdfauthor = {},
  pdfkeywords = {mathematics},
  pdfsubject = {},
  pdfpagemode = UseNone
}
\graphicspath{{./pic/}}
\theoremstyle{plain}
\newtheorem{lemma}{Lemma}[section]
\newtheorem{theorem}[lemma]{Theorem}

\theoremstyle{definition}
\newtheorem{definition}[lemma]{Definition}

\theoremstyle{remark}
 
\newtheorem{problem}{Open Problem}

\newcommand{\R}{\mathbbm{R}}
\newcommand{\Q}{\mathbbm{Q}}
\newcommand{\Z}{\mathbbm{Z}}
\DeclareMathOperator{\aff}{\mathrm{aff}}

\DeclareMathOperator{\boundary}{bd}

\newcommand{\fset}{\mathcal{F}}
\newcommand\Fset[2]{\fset^{#1}_{#2}}
\newcommand\PiFset[3]{\Pi_{#2}(\fset^{#1}_{#3})}
                
\begin{document}

\title{Semi-algebraic sets of $f$-vectors}

\author[Sj\"oberg]{Hannah Sj\"oberg} 
\address{Institute of Mathematics, Freie Universität Berlin, Arnimallee 2, 14195 Berlin, Germany}
\email{sjoberg@math.fu-berlin.de}

\author[Ziegler]{G\"unter M. Ziegler} 
\address{Institute of Mathematics, Freie Universität Berlin, Arnimallee 2, 14195 Berlin, Germany}
\email{ziegler@math.fu-berlin.de}

\thanks{Research by the second author was supported by the DFG Collaborative Research Center TRR~109 ``Discretization in Geometry and Dynamics.''}
\date{October 18, 2018}

\thanks{Dedicated to Anders Bj\"orner on occasion of his 70th birthday}

\begin{abstract}\noindent%
Polytope theory has produced a great number of remarkably simple and complete 
characterization results for face-number sets or $f$-vector sets of classes of polytopes.
We observe that in most cases these sets can be described as the
intersection of a semi-algebraic set with an integer lattice.
Such \emph{semi-algebraic sets of lattice points} have not received 
much attention,
which is surprising in view of a close connection to Hilbert's Tenth problem, which deals with their
projections. 

We develop proof techniques in order to show that, despite the observations
above, some
$f$-vector sets are \emph{not} semi-algebraic sets of lattice points.
This is then proved for 
the set of all pairs $(f_1,f_2)$ of $4$-dimensional polytopes,
the set of all $f$-vectors of simplicial $d$-polytopes for $d\ge6$, and 
the set of all $f$-vectors of general $d$-polytopes for $d\ge6$.
For the $f$-vector set of all $4$-polytopes this remains open.
%
%
\end{abstract}

\maketitle

\section{Introduction}

For any $d\ge1$, let $\Fset{d}{}\subset\Z^d$ denote the set of all $f$-vectors
$(f_0,f_1,\dots,f_{d-1})$ of $d$-dimensional polytopes.
Thus $\Fset{1}{} =\{2\}\subset\Z$ and $\Fset{2}{}=\{(n,n):n\ge3\}\subset\Z^2$. 
In 1906, Steinitz \cite{Stei3} characterized the set $\Fset{3}{}$ of $f$-vectors
$(f_0,f_1,f_2)$ of $3$-dimensional polytopes $P$ as 
\[
\Fset{3}{} = 
\big\{(f_0,f_1,f_2)\in \Z^3 : 
f_0 -f_1 + f_2= 2,\ 
f_2 \le 2f_0-4,\
f_0 \le 2f_2 -4\big\}.
\]
Thus for $d\le3$ the set $\Fset{d}{}\subset\Z^d$ has a very simple structure: 
It is the set of all integer points in a $(d-1)$-dimensional rational cone.
 
Inspired by this, Grünbaum in 1967 \cite[Sec.~10.4]{Gr1-2}
and subsequently Barnette and Reay characterized the sets
$\PiFset{4}{ij}{}$ of all pairs $(f_i,f_j)$
that occur for $4$-dimensional polytopes.
Again they got complete and reasonably simple answers: 
They found that in all cases 
this is the set of all integer points between some fairly obvious
upper and lower bounds, with finitely many exceptions.

In our work here we start with a formal definition of what we mean by
a ``simple answer'': 

\begin{definition}[Semi-algebraic sets of integer points]\label{def:semi-algebraic-integer}
A set of $A\subset\Z^d$ is a \emph{semi-algebraic set of integer points}
if it is the set of all integer points in a semi-algebraic set, 
that is, if $A=S\cap\Z^d$ for some semi-algebraic set $S\subseteq\R^d$.
\end{definition}

For this recall that a \emph{basic semi-algebraic set} is a 
subset $S\subseteq\R^d$ that can be defined by a finite 
number of polynomial equations and inequalities.
A \emph{semi-algebraic set}
is any finite union of basic semi-algebraic sets.
The semi-algebraic set is \emph{defined over $\Z$} if the polynomials can be 
chosen with integral coefficients. In this case we will call this
a \emph{$\Z$-semi-algebraic set}.
See Basu, Pollack \& Roy \cite{basu03:_algor_real_algeb_geomet} for 
background on semi-algebraic sets.

It turns out that Definition~\ref{def:semi-algebraic-integer}
is not quite general enough for $f$-vector theory,
as we need to account for modularity constraints
that may arise due to projections. For example, 
$A:=\{(x,y)\in\Z^2:x=2y\}$
is a semi-algebraic set of integer points, 
but its projection to the first coordinate $\Pi_1(A)=2\Z$ is not
if we insist that the lattice is $\Z$.
This is relevant for $f$-vector sets, as for example
every simplicial $3$-polytope satisfies $3f_2=2f_1$, so $f_2$ is even and $f_1$
is a multiple of $3$. Consequently $\PiFset{3}{2}{s}=\{4,6,8,\dots\}$,
the set of all possible facet numbers of simplicial $3$-polytopes, 
is \emph{not} a semi-algebraic set of integer points,
but it is a semi-algebraic set of lattice points:  

\begin{definition}[Semi-algebraic sets of lattice points]\label{def:semi-algebraic-lattice}
A subset $A\subset\R^d$ is a \emph{semi-algebraic set of lattice points}
if it is an intersection set of a semi-algebraic set 
with an affine lattice,
that is, if $A=S\cap\Lambda$ for a suitable semi-algebraic set $S\subseteq\R^d$
and an affine lattice $\Lambda\subset\R^d$. 
\end{definition}

Here by an \emph{affine lattice} we mean any translate of a linear lattice,
that is, a discrete subset $\Lambda\subset\R^d$
that is closed under taking affine combinations
$\lambda_1a_1+\dots+\lambda_na_n$ for $n\ge1$ with $\lambda_1,\dots,\lambda_n\in\Z$
and $\lambda_1+\dots+\lambda_n=1$.
We will only consider \emph{integer lattices}, that is, sublattices $\Lambda\subseteq\Z^d$.
Moreover, without loss of generality we may always assume that the lattice
is the affine lattice $\Lambda=\aff_\Z A$ spanned by $A$: The set of all affine
combinations yields a lattice if $A\subset\Z^d$, and the lattice $\Lambda$
has to contain $\aff_\Z A$. 

With the generality of Definition~\ref{def:semi-algebraic-lattice},
a great number of characterization results achieved in the $f$-vector theory
of polytopes imply that full $f$-vector sets or coordinate projections
(that is, single face numbers or face number pairs) are
semi-algebraic sets of lattice points.
We will summarize this in Section~\ref{sec:examples}.  

Semi-algebraic sets of lattice points $A\subset\Z^d$ are easy to
identify and to characterize for $d=1$; see the beginning of Section~\ref{sec:techniques}. 
However, already for sets in the plane $A\subset\Z^2$
this becomes non-trivial. For example, the answer depends on the
field of definition: The set $\{(x,y)\in\Z^2:y\ge\pi x\}$ 
is an $\R$-semi-algebraic set of integer points, 
but not a $\Z$-semi-algebraic set of lattice points.

Our two main results are the following:

\begin{theorem}\label{thm:Pi4_12} 
The set $\PiFset{4}{12}{}$ of pairs $(f_1,f_2)$ for $4$-dimensional polytopes
is not an $\R$-semi-algebraic set of lattice points. 
\end{theorem}

\begin{theorem}\label{thm:Fd}
For any $d\ge 6$,
the set $\Fset{d}{}$ of all $f$-vectors of $d$-dimensional polytopes is 
not an $\R$-semi-algebraic set of lattice points. 
\end{theorem}

In Section~\ref{sec:techniques} we develop proof techniques, including
the ``\nameref{lemma:strip}.'' Based on this,
the proof of Theorem~\ref{thm:Pi4_12} is given in Section~\ref{sec:dim4} and
the proof of Theorem~\ref{thm:Fd} in Section~\ref{sec:dim6}.%
\medskip

\noindent
Coordinate projections of semi-algebraic sets of lattice points
are not in general again semi-algebraic.
Indeed, Matiyasevich's Theorem \cite{Matiyasevich} (see Davis \cite{Davis:Hilbert10}
and Matiyasevich \cite{Matijasevic}) 
states that all recursively enumerable sets of integer points are Diophantine sets. 
Matiyasevich proved that every Diophantine set has \emph{rank} 
at most $9$ (see \cite{Jones}, and most recently Sun \cite{Sun:Hilbert10}), that is,
it is the projection of the integer points of
some semi-algebraic set defined over $\Z$ with at most $9$ additional variables.
Thus, in particular, $f$-vector sets of polytopes, like
$\Fset{d}{}$ and $\Fset{d}{s}$ for $d\ge2$, are Diophantine of rank at most~$9$,
since Grünbaum has noted in \cite[Sect.~5.5]{Gr1-2} that such sets 
are recursively enumerable.

On the other hand, 
the semi-algebraic sets of lattice points that we consider in this paper are more restrictive
than Diophantine sets: 
We are interested in the cases when a set cannot be described as the set of integer points of a semi-algebraic set
(defined over $\Z$ or $\R$) without additional variables.
With the proof of Theorem~\ref{thm:Pi4_12},
we will see that, for example, the set $\PiFset{4}{12}{}$ 
can be described using one additional variable (if we allow for 
inequalities, which is not usual in the context of Diophantine sets,
but equivalent).
The same is true for the $f$-vector set of simplicial $6$-polytopes, $\Fset{6}{s}$.

In summary, we will show that many $f$-vector sets are semi-algebraic
(Section~\ref{sec:examples}), while some are not (Theorems \ref{thm:Pi4_12} and \ref{thm:Fd}).
The crucial question that remains open concerns dimensions $4$ and~$5$:

\begin{problem}
	Is the $f$-vector set of $4$-polytopes $\Fset{4}{}\subset\Z^4$ semi-algebraic?
\end{problem}

(The size/fatness projection of $\Fset{4}{}$ displayed and discussed 
by Brinkmann \& Ziegler~\cite{Z157} suggests that the answer is no.)

\begin{problem}
Is the $f$-vector set of $5$-polytopes $\Fset{5}{}\subset\Z^5$ semi-algebraic?
\end{problem}

We refer to Grünbaum \cite[Chap.~8-10]{Gr1-2} and Ziegler \cite[Lect.~8]{Z35} for
general information and further references on polytopes and their $f$-vectors.
The lattices spanned by $f$-vector sets, as well as more general additive
(semi-group) structures on them, are discussed in Ziegler \cite{Z165}.

\section{Semi-algebraic sets of \texorpdfstring{$\boldsymbol{f}$}{f}-vectors}\label{sec:examples}

The following theorem summarizes a great number of works in $f$-vector theory,
started by Steinitz in 1906 \cite{Stei3} and re-started by Grünbaum in 1967 \cite[Chap.~8-10]{Gr1-2}.  

\begin{theorem}
The following sets of face numbers, face number pairs, and $f$-vectors,
are $\Z$-semi-algebraic sets of lattice points:
\begin{enumerate}[\rm(i)]
\item \label{thm:part-01 d le 3} 
	$\Fset{d}{}$, the set of $f$-vectors of $d$-dimensional polytopes for $d\le 3$,
\item \label{thm:part-02 d le 3 simplicial} 
	$\Fset{d}{s}$ and $\Fset{d}{s*}$, the sets of $f$-vectors of 
	simplicial and of simple $d$-dimensional polytopes for $d\le 5$, 
\item \label{thm:part-03 cs 3-polytopes}
	$\Fset{3}{cs}{}$, the set of $f$-vectors of $3$-dimensional centrally-symmetric polytopes, 
\item \label{thm:part-04 i-faces not diagonal}
	$\PiFset{d}{i}{}$, the sets of numbers of $i$-faces of $d$-polytopes 
	for all $d$ and $i$, 	
\item \label{thm:part-06 i-faces simplicial}
	$\PiFset{d}{i}{s}$ and~$\PiFset{d}{i}{s*}$,
	the sets of numbers of\/ $i$-faces of simplicial and of simple $d$-polytopes,
\item \label{thm:part-07 2dim 4-polytopes}
	$\PiFset{4}{01}{}$, $\PiFset{4}{02}{}$, and $\PiFset{4}{03}{}$,
	sets of face number pairs of $4$-polytopes,
\item \label{thm:part-08 01 5-polytopes}
	$\PiFset{5}{01}{}$,
	the set of pairs of ``number of vertices and number of edges'' for \mbox{$5$-polytopes},	
\item \label{thm:part-09 f0 cubical d-polytopes}
	$\PiFset{d}{0}{cub}$, the set of vertex numbers of cubical $d$-polytopes, 
	for $d\le4$ and for all even dimensions $d$, 
\item \label{thm:part-10}
	$\PiFset{4}{0}{2s2s}$, the set of vertex numbers of $2$-simplicial $2$-simple $4$-polytopes,
	and
\item \label{thm:part-11}
    $\PiFset{d}{0,d-1}{}$, the set of pairs of ``number of vertices and number of facets''
	of $d$-polytopes, for even dimensions $d$. 
\end{enumerate}
\end{theorem}

\begin{proof}
In each case, the set in question is described as all the integers or integer points
that satisfy a number of polynomial equations, strict inequalities, non-strict inequalities,
or inequalities:
\smallskip

\noindent
(\ref{thm:part-01 d le 3})
This is Steinitz's result \cite{Stei3}, as quoted in the introduction.
In this case, the equation and inequalities are linear.
It also includes the information that the $f$-vector set of
simplicial $3$-polytopes is
$\Fset{3}{s}=\{(n,3n-6,2n-4):n\ge4\}$, which yields
the case $d=3$ of
 (\ref{thm:part-02 d le 3 simplicial}).
\smallskip

\noindent%
(\ref{thm:part-02 d le 3 simplicial})
$\Fset{4}{s}$ and $\Fset{5}{s}$ can be deduced from the $g$-Theorem (see Section~\ref{sec:dim6}):
\begin{align*}
\Fset{4}{s}=&\{(f_0,f_1,-2f_0+2f_1,-f_0+f_1)\in\Z^4 : 
f_0 \ge 5, \ 4f_0-10\le f_1\le\tfrac 1 2 f_0(f_0-1)\},\\
\Fset{5}{s}=&\{(f_0,f_1,-10f_0+4f_1+20,-15f_0+5f_1+30, -6f_0+2f_1+12)\in\Z^5 : \\
&\hphantom{\{(f_0,f_1,-2f_0+2f_1,-f_0+f_1)\in\Z^4 : \mbox{}}
f_0 \ge 6, \ 5f_0-15\le f_1\le\tfrac 1 2 f_0(f_0-1)\}.
\end{align*}

\noindent%
(\ref{thm:part-03 cs 3-polytopes})
The $f$-vector set $\Fset{3}{cs}$ of centrally-symmetric $3$-polytopes
spans the lattice $(2\Z)^3$. 
Werner {\cite[Thm.~3.3.6]{Werner-diss}} has 
described it as
\[\Fset{3}{cs}=\{(f_0,f_1,f_2)\in(2\Z)^3 : 
f_0 -f_1 + f_2= 2, f_2 \le2f_0-4, f_0 \le 2f_2 -4, f_0 + f_2 \ge 14\}.\]

\noindent%
(\ref{thm:part-04 i-faces not diagonal}),(\ref{thm:part-06  i-faces simplicial})
Bj\"orner \& Linusson \cite{BjornerLinusson} showed that
for any  integers $0\le i<d$ there are numbers $N(d,i)$ and $G(d,i)$ 
such that there is a simple $d$-polytope 
with $n> N(d,i)$ $i$-faces if and only if $n$ is a multiple of~$G(d,i)$.
Additionally, $G(d,i)=1$ for $i\ge \big\lfloor\tfrac{d+1}{2}\big\rfloor$.
As a consequence the $1$-dimensional coordinate projection of the set of $f$-vectors 
of all simple $d$-polytopes is a semi-algebraic
set of integer points over $\Z$.
The same holds for simplicial polytopes, by duality.
Here, the number $G(d,i)$ is equal to $1$ for $i\le \big\lfloor\tfrac{d+1}{2}\big\rfloor -1$.
This also implies that the $1$-dimensional projection sets $\PiFset{d}{i}{}$ are semi-algebraic sets of integer points
for all choices of $d$ and $i$ with the possible exception of odd $d$ and 
$i=\tfrac{d-1}{2}$.
In order to show that $\PiFset{2i+1}{i}{}$ is a semi-algebraic set of integer points, 
we derive from \cite{BjornerLinusson}:
\[
	G(2i+1,i) =  \begin{cases*}
                    p & if $i+2=p^s$ for some integer $s\ge 1$ and some prime $p$, \\
                    1 & otherwise.
                 \end{cases*} 
\]
Hence, $\PiFset{2i+1}{i}{}$ is a semi-algebraic set of integer points
if $i+2\neq p^s$ for all primes $p$ and all integers $s$. 

Let now $i+2=p^s$ for some $s\ge 1$ and a prime~$p$.
Assume that we have a $(2i+1)$-polytope $P$ with a simplex facet
such that $\gcd(f_i(P),p)=1$.
Then using the construction of connected sums 
by Eckhoff \cite{Eckhoff06} (see also \cite[p.~274]{Z35})
to successively add copies of $P$, its dual $P^*$, simple and simplicial 
polytopes, we obtain $(2i+1)$-polytopes with all possible numbers $n$ of $i$-faces
for all sufficiently large $n$, that is, for $n\ge M(d,i)$. 

To complete our proof, we give a construction of the polytope $P$.
We consider two different cases.
In the first case, let $i+2=2^s$ for some $s\ge 2$.

Since $G(2i,i)=1$ we can find a simple $2i$-polytope $R$ with an odd number of $i$-faces.
From  \cite{BjornerLinusson} we get that
\[
     G(2i,i-1) = \begin{cases*}
                    2 & if $i+2=2^t$ for some integer $t$, \\
                     1 & otherwise.
                 \end{cases*}
\]
Thus $R$ has an even number of $(i-1)$-faces.
Let $Q$ be the connected sum $R\#R^*$ of $R$
and its dual. Then $f_i(Q)=f_i(R)+f_i(R^*)=f_i(R)+f_{i-1}(R)$
is odd and $Q$ has a simplex facet.

Let now $P$ be the bipyramid over $Q$.
Then $f_i(P)= 2f_{i-1}(Q) +f_i(Q)$ is odd and $P$ has a simplex facet.

In the second case, $i+2=p^s$ for some integer $s\ge 1$ and some odd prime $p$.
Choose a simple $2i$-polytope $R$ with $f_0(R)\ge i+1$ and $\gcd(f_i(R),p)=1$.
Such a polytope $R$ exists since $G(2i,i) = 1$.
 Let $P_1$ be the prism over $R$ and $P_2$ the pyramid over $R^*$. 
 Then $f_i(P_1)=f_{i-1}(R)+2f_i(R)$,
$f_i(P_2)=f_{i-1}(R)+f_i(R)$,  $P_1$ is a simple polytope and $P_2$ has $f_0(R)\ge i+1$ simplex facets.
Let $P$ be the connected sum of $P_2$ and $i+1$ copies of $P_1$:
\[
	P=({\cdot}{\cdot}((P_2\# \underbrace{P_1)\#P_1)\#\dots P_1)\#P_1}_{i+1}. 
\]
\enlargethispage{4mm}%
The resulting $(2i+1)$-polytope $P$ has a simplex facet and
\begin{align*}
	f_i(P)&=f_i(P_2)+(i+1)f_i(P_1)\\
	        &=f_{i-1}(R)+f_i(R)+(i+1)(f_{i-1}(R)+2f_i(R))\\
	        &=(i+2)(f_{i-1}(R)+2f_i(R))-f_i(R)\\
	        &=p^s(f_{i-1}(R)+2f_i(R))-f_i(R),
\end{align*}
which is coprime to $p$, since $f_i(R)$ is coprime to $p$. 
\smallskip

\noindent%
(\ref{thm:part-07 2dim 4-polytopes})
The $2$-dimensional coordinate projections $\PiFset{4}{ij}{}$  
have been characterized by Gr\"unbaum \cite[Thm.~10.4.1, 10.4.2]{Gr1-2},
Barnette \cite{barnette74:_e_s}, and Barnette \& Reay \cite{barnette73:_projec}:
$\PiFset{4}{03}{}$ consists of all the integer points between two parabolas,
$\PiFset{4}{01}{}$ is the set of all integer points between a line and a parabola,
			with four exceptions, and
$\PiFset{4}{02}{}$ is the set of all integer points between two parabolas,
		except for the integer points on an exceptional parabola,
		and ten more exceptional points.
\smallskip

\noindent%
(\ref{thm:part-08 01 5-polytopes})
The set $\PiFset{5}{01}{}$ was recently determined independently 
by Kusunoki \& Murai~\cite{KusunokiMurai} and by
Pineda-\/Villavicencio, Ugon \& Yost~\cite{PinedaUgonYost}: 
		It is the set of all integer points between a line and a parabola,
		except for the integer points on two lines and three more exceptional points.
\smallskip

\noindent	
(\ref{thm:part-09 f0 cubical d-polytopes})
The possible vertex numbers of cubical $3$-polytopes are
$\PiFset{3}{0}{cub}=\{8\}\cup\{n\in\Z:n\ge10\}$.
Blind \& Blind \cite{blind94:_gaps_i} proved that the number
of vertices $f_0$ as well as of edges $f_1$ are even for every cubical $d$-polytope
if $d\ge4$ is even.
According to Blind \& Blind \cite[Cor.~1]{BlBl2},
there are ``elementary'' cubical $d$-polytopes $C^d_k$ with $2^{d+1}-2^{d-k}$ vertices,
for $0\le k<d$. (In particular, $C^d_{d-1}$ has $2^{d+1}-2$ vertices.)
As the facets of these polytopes are projectively equivalent to standard cubes,
we can glue them in facets (as in Ziegler \cite[Sect.~5.2]{Z165}),
and thus obtain all sufficiently large even vertex numbers.
Thus $\PiFset{d}{0}{cub}$ is a semi-algebraic subset of the lattice $2\Z$
for even $d\ge4$.
\smallskip

\noindent	
(\ref{thm:part-10})
Paffenholz \& Werner \cite{PaffenholzWerner:many} and Miyata \cite{Miyata-diss} 
proved that the set of possible numbers of vertices for $2$-simplicial $2$-simple $4$-polytopes
is $\PiFset{4}{0}{2s2s}=\{5\}\cup\{n\in\Z:n\ge9\}$.
\smallskip

\noindent	
(\ref{thm:part-11})
For even $d$ and $n+m\ge\binom{3d+1}{\lfloor d/2\rfloor}$
there exists a $d$-polytope $P$ with $n$ vertices and $m$ facets
if and only if 
$m\le f_{d-1}(C_d(n))$
and
$n\le f_{d-1}(C_d(m)),$
where $C_d(n)$ denotes the $d$-dimensional cyclic polytope with $n$ vertices \cite{Z160}.
\end{proof}

\section{Proof Techniques}\label{sec:techniques}

It is easy to see that a subset $A\subseteq\Z$
is a semi-algebraic set of integer points if and only if
it consists of a finite set of (possibly unbounded)
intervals of integer points.
Equivalently, a subset $A\subseteq\Z$ is \emph{not}
a semi-algebraic set of integer points
if and only if there is a strictly monotone (increasing or decreasing)
infinite sequence of integers, with $a_1<a_2<\cdots$ or $a_1>a_2>\cdots$,
such that $a_{2i}\in S$ and $a_{2i+1}\in\Z{\setminus}S$.

The same characterization holds for semi-algebraic sets of lattice points
$A\subset\R$, where $\aff_\Z A$ takes over the role of the integers~$\Z$.

Examples of subsets of $\Z$ that are {not} $\R$-semi-algebraic sets of 
lattice points include 
the set of squares $\{ n^2: n\in \Z_{\ge0}\}$, 
the set $\{n \in \Z : n \not\equiv 0 \mod 3\}$,
and the set $\{1,2,4,6,8,10,\ldots\}$.

For subsets of $\Z^2$, or of $\Z^d$ for $d>2$, we do not
have -- or expect -- a complete characterization of semi-algebraic
sets of integer points. 

There are some obvious criteria: For example, every finite set 
of integer points is semi-algebraic, finite unions of semi-algebraic
sets of lattice points with respect to the same lattice are semi-algebraic, products of semi-algebraic
sets of lattice points are semi-algebraic, and so on.

However, these simple general criteria turn out to be of little use 
for studying the specific sets of integer points we are interested in.
The ``finite oscillation'' criterion of the one-dimensional case
suggests the following approach for subsets $A\subset\Z^d$:

\begin{lemma}[Curve lemma]\label{lemma:curve}
If there is a semi-algebraic curve $\Gamma$ that 
along the curve contains an infinite sequence of integer 
points $a_1,a_2,\dots$ (in this order along the curve) with 
$a_{2i}\in\Gamma\cap A$ and $a_{2i+1}\in\Gamma{\setminus}A$,
then $A$ is \emph{not} a semi-algebraic set of integer points.
Similarly, if this holds with $a_1,a_2,\ldots\in\Lambda:=\aff_\Z A$,
then $A$ is \emph{not} a semi-algebraic set of lattice points.
\end{lemma}

However, for our examples the semi-algebraic curves $\Gamma$ of
Lemma~\ref{lemma:curve} do not exist.
Thus to show that a $2$-dimensional set is not a semi-algebraic set of lattice points 
we develop a better criterion:
Instead of the ``curve lemma'' we rely on a ``strip lemma,''
which in place of single algebraic curves considers strips generated by
disjoint translates of an algebraic curve.

In the following, 
we refer to Basu, Pollack and Roy \cite{basu03:_algor_real_algeb_geomet} 
for notation and information about semi-algebraic sets.

\begin{definition}\label{def:strip}
Let $\gamma_0=\{(x,f(x)):x\ge0\}\subset \R^2$ be a curve,
where $f(x)$ is an algebraic function defined for all $x\ge 0$, 
and let $c$ be a vector in $\R^2$.
If the translates $\gamma_t=\gamma_0+tc$ for $t\in [0,1]$ are disjoint,
then we refer to this family of curves $\mathcal{C}\coloneqq \{\gamma_t\}_{t\in[0,1]}$ 
as a \emph{strip of algebraic curves}.
A \emph{substrip} of $\mathcal{C}$ is a family $\mathcal{C}_J$ of all curves $\gamma_t$ with $t\in J$,
where $J$ is any closed interval $J\subseteq [0,1]$ of positive length.
\end{definition}

\begin{lemma}[Strip lemma]\label{lemma:strip}
Let $L \subset\Z^2$ be a set of integer points and
$\Lambda =\aff_{\Z}L$ the affine lattice spanned by $L$.
If there exists a strip of algebraic curves $\mathcal{C}$ 
such that every substrip~$\mathcal{C}_J$ contains infinitely many
points from $\Lambda\cap L$ and infinitely many points from $\Lambda\setminus L$,
then $L$ is not an $\R$-semi-algebraic set of lattice points.
\end{lemma}

See Figure~\ref{plot:strip_lemma} for a visualization.

\begin{figure}[ht!]
  \centering
  \includegraphics[scale=.33]{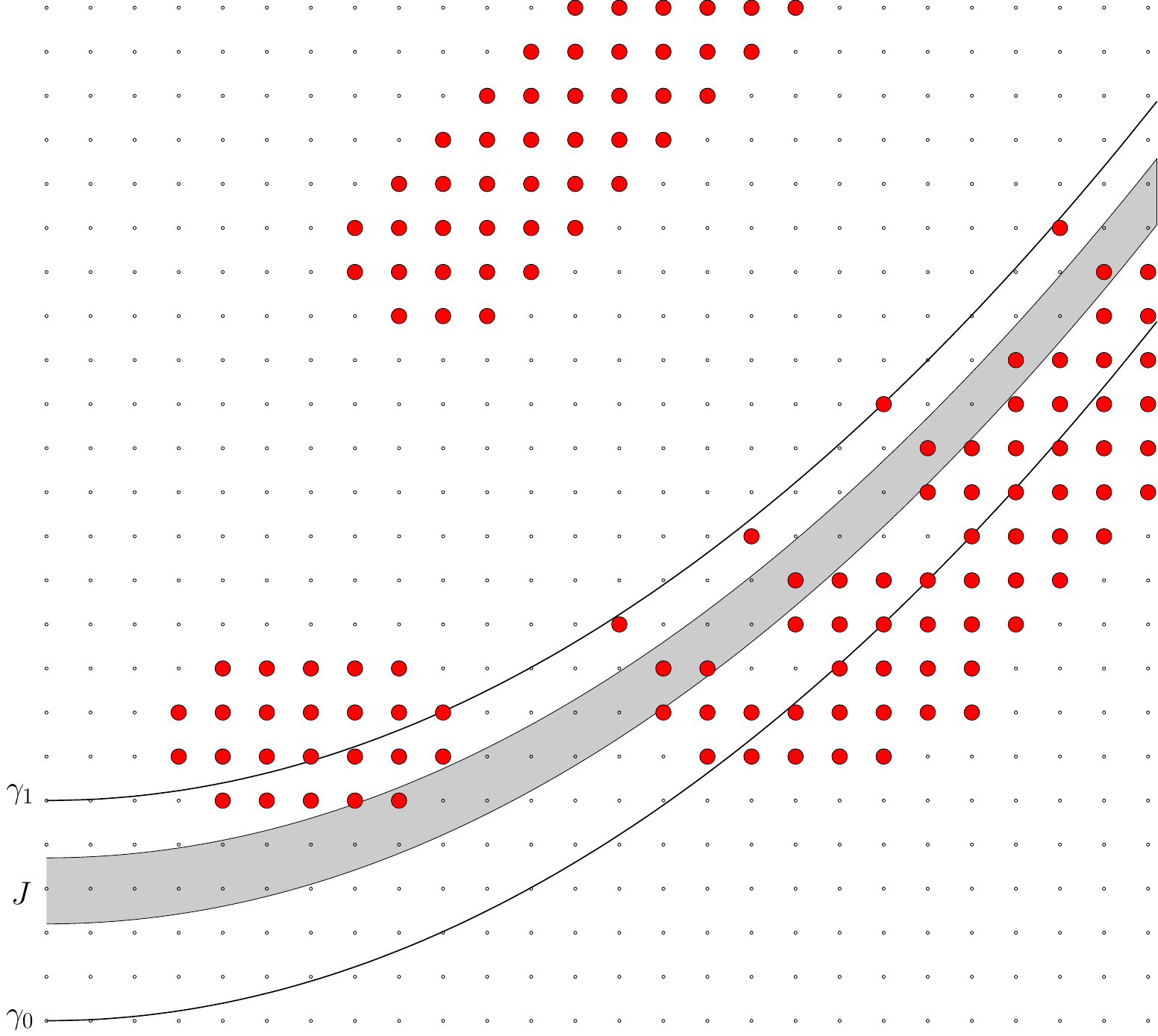}
\caption{%
This sketch illustrates that for
a semi-algebraic set $L$ of lattice points there cannot 
be an infinite sequence of lattice points in $L$, as well as not in~$L$,
in every substrip between $\gamma_0$ and~$\gamma_1$.}
\label{plot:strip_lemma}
\end{figure}

\begin{proof}

Assume that $L$ is $\R$-semi-algebraic,
that is, there exists an $\R$-semi-algebraic set $S\subset \R^2$ such that 
$L=S\cap\Lambda$.
The boundary of $S$ is the intersection of the closure of $S$ 
with the closure of $\R^2 \backslash S$, 
$\boundary (S)=\overline{S}\cap \overline{\R^2 \backslash S}$.
The Tarski--Seidenberg theorem yields that the closure of a semi-algebraic set in $\R^d$ is again a semi-algebraic set \cite[Prop.~3.1]{basu03:_algor_real_algeb_geomet}. 
The boundary $\boundary(S)$ is the intersection of two semi-algebraic sets and hence itself a semi-algebraic set.

Any semi-algebraic set consists of finitely many connected components, 
all being semi-algebraic \cite[Thm.~5.19]{basu03:_algor_real_algeb_geomet}.

From this we want to derive that
for any strip of algebraic curves $\mathcal{C}$ 
there exists a substrip $\mathcal{C}_J$ of $\mathcal{C}$
such that for some $n\ge 0$, all lattice points $(a,b)\in \Lambda$ with $a\ge n$ in the substrip
belong entirely to $L$, or all of them do not belong to $L$.

Denote by $\beta_1, \ldots, \beta_m$ all those connected components of $\boundary(S)$ that contain points $(x,y)\in \R^2$
with arbitrarily large $x$ in a strip $\mathcal{C}$.
If such components do not exist, then either all points of $\mathcal{C}\cap \Lambda$ with sufficiently large $x$-coordinate
(that is, all but finitely many of these points) lie in $L$,
or all of them do not lie in~$L$.

The intersection of a semi-algebraic 
component $\beta_j$ and any semi-algebraic curve $\gamma_t$
is again semi-algebraic, so it consists of finitely many connected components.
Thus for any given $\beta_j$ and $\gamma_t$, 
$\beta_j$ has finitely many 
branches to infinity such that each branch
eventually (for all sufficiently large $x$-coordinates)
stays above $\gamma_t$, or below $\gamma_t$, or on $\gamma_t$.
Thus by continued bisection we find that 
there exists some value $n'\ge 0$ such that the restriction of each $\beta_j$ 
to $x\ge n'$ has finitely many components, each of which either lies on a curve $\gamma_t$, or it is a curvilinear asymptote to \emph{some} curve $\gamma_t$.
Let the components of $\{(x,y)\in\beta_j:x\ge n'\}$ be asymptotic to 
(or lie on)
$\gamma_{t_1},\dots, \gamma_{t_k}$, with $0\le t_1\le\dots\le t_k\le 1$,
and let $[\delta_0,\delta_1]\subset [0,1]$ be an interval interval of positive length
(that is, with $\delta_0<\delta_1$) that is disjoint from $\{0,t_1,\dots, t_k,1\}$.
Then there exists an $n\ge 0$ such that
the lattice points $(a,b)\in \Lambda$ with $a\ge n$ contained in the
substrip obtained from $[\delta_0,\delta_1]$ 
either all belong to $L$ or they all do not belong to~$L$.
\end{proof} 

\section{Edge and ridge numbers of 4-polytopes}\label{sec:dim4}
In this section we prove Theorem~\ref{thm:Pi4_12}.

\begin{theorem}[Barnette {\cite[Thm. 1]{barnette74:_e_s}}, see also \cite{Z59}, with corrections] \label{f1f2} 
Let $f_1$ and $f_2$ positive integers with $f_1 \ge f_2$. 
Then there is a $4$-polytope $P$ with $f_1(P) = f_1$
and $f_2(P) = f_2$ if and only if 

\begin{eqnarray*}
f_2 &\ge& \tfrac1 2 f_1 +\big\lceil\sqrt{f_1+\tfrac 9 4}+\tfrac 1 2\big\rceil+1,\\
f_2 &\neq& \tfrac1 2 f_1 +\sqrt{f_1+\tfrac {13} 4}+2,
\end{eqnarray*}
and $(f_1,f_2)$ is not one of the $13$ pairs
\begin{align*}
&(12,12),(14,13),(14,14),(15,15),(16,15),(17,16),(18,16),\\
&(18,18),(20,17),(21,19),(23,20),(24,20),(26,21).
\end{align*}
\end{theorem}

The case when $f_1(P) \le f_2(P)$ is given by duality.
See Figure~\ref{plot:f1f2}.

\begin{figure}[ht!]
  \centering
  \includegraphics[scale=.4]{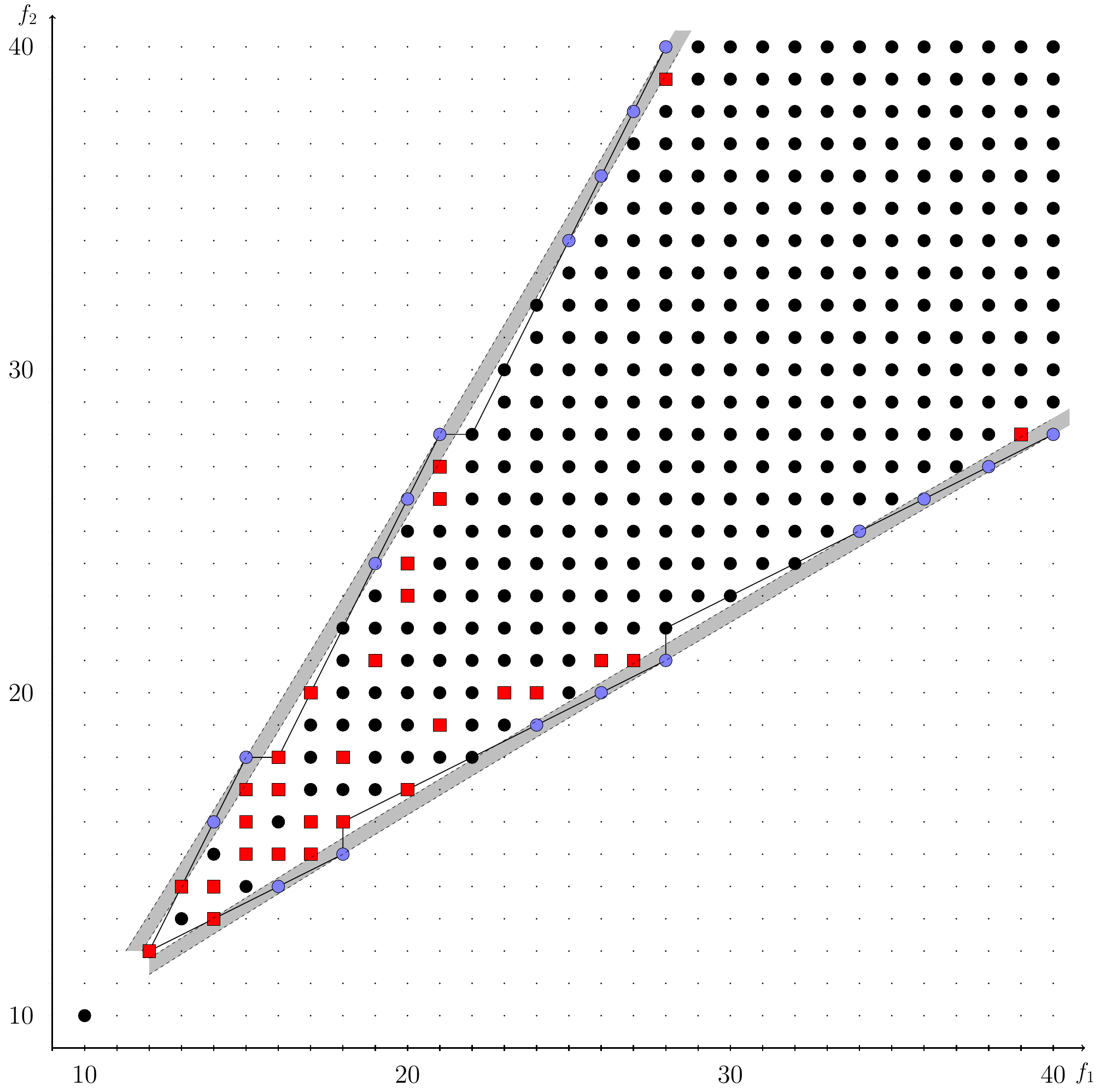}
\caption{The set $\Pi^4_{12}$ with the two strips that will play a crucial role in the proof
that the set is not semialgebraic, see Lemma~\ref{gammacurves} and its proof.}
\label{plot:f1f2}
\end{figure}

Now we show that there is no semi-algebraic description of the set
of pairs $(f_1,f_2)$ by proving that the set 
\begin{equation}
A \ \coloneqq\  
\big\{(x,y) \in \Z^2 : 
 x\ge 0, y\ge \tfrac x 2 +
\big\lceil\sqrt{x+\tfrac 9 4}+\tfrac 12\big\rceil
 +  1 \big\}\label{def:setA}
\end{equation}
is not a semi-algebraic set of lattice points. See Figure~\ref{areaA}.

\begin{figure}[htb]
  \centering
  \includegraphics[width=13cm]{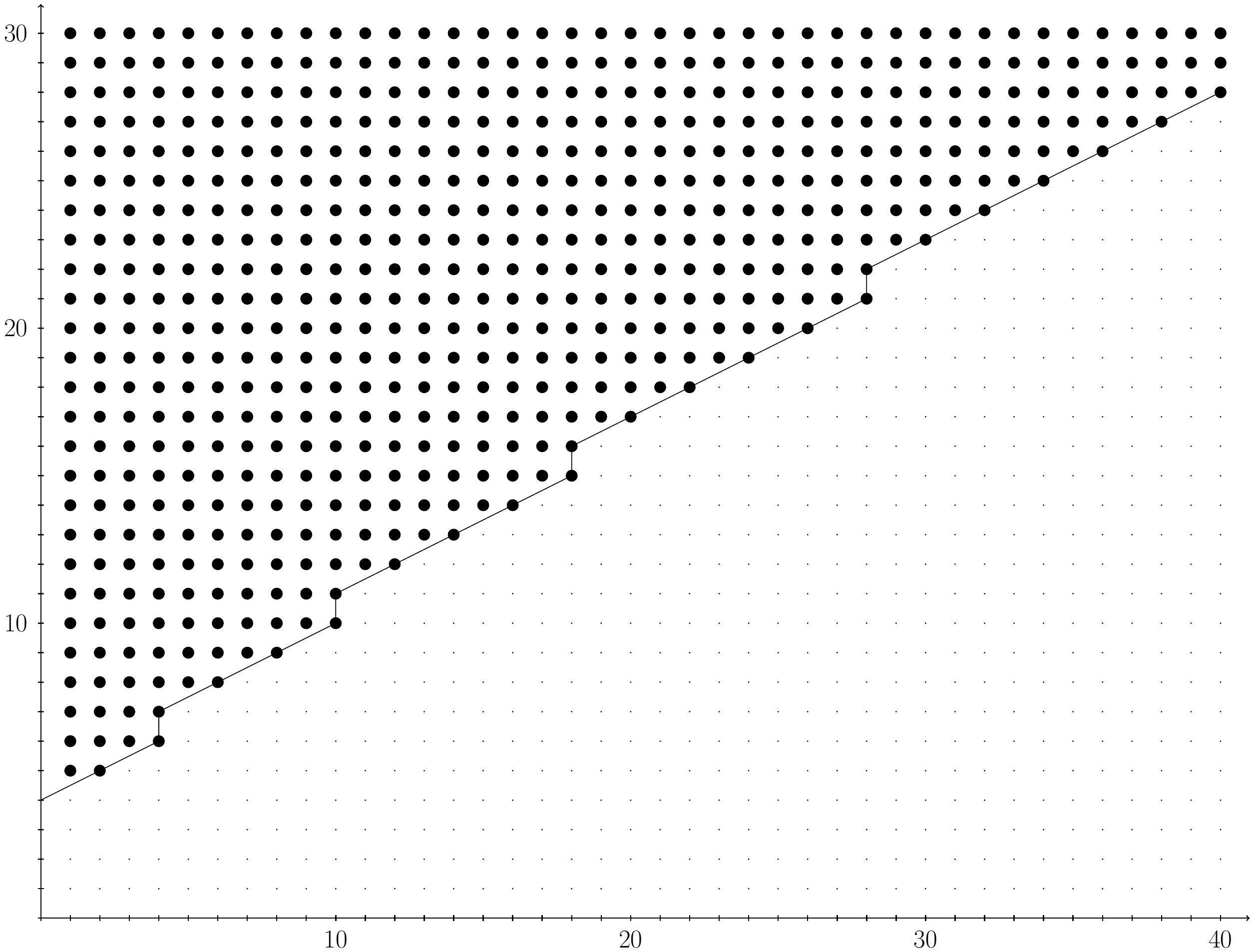}
\caption{The set 
$ A=
\big\{(x,y) \in \Z^2 : 
 x\ge 0, y\ge \tfrac x 2 +
\big\lceil\sqrt{x+\tfrac 9 4}+\tfrac 1 2\big\rceil + 1 \big\} $}
\label{areaA}
\end{figure}

The proof strategy is the following:
In Lemma~\ref{lemma_curves} we give an alternative description of the set.
In Lemma~\ref{gammacurves} we observe that our set has the property described in 
Lemma~\ref{lemma:strip}, which implies that the set is not an $\R$-semi-algebraic set of lattice points.

\begin{lemma}\label{lemma_curves}
Let $x$ and $y$ be nonnegative integers. 
Then  
\begin{equation}
y\ge \tfrac x 2 +\big\lceil\sqrt{x+\tfrac 9 4}+\tfrac 1 2\big\rceil+1\label{cond1}
\end{equation}
if and only if  
\begin{eqnarray}
      y & \ge & \tfrac x 2 +\sqrt{x+\tfrac 9 4}+ 2  \notag\\ 
	  \text{or}\hspace{35mm} && 	            \notag\\[-14mm] 
      y &  =  & \tfrac x 2 +\sqrt{x+\tfrac 9 4}+\tfrac 3 2+r 
	  \hspace{55mm}\left.\rule[-14mm]{0pt}{12mm}\right\}\hspace{5mm}  \label{cond2}\\[-14mm] 
	    &     & \text{for some }r=i+\tfrac 1 2-\sqrt{(i+\tfrac 1 2)^2-2j} \notag\\ 
        &     & \text{with }i,j \in \Z,\ i\ge 1,\ 0\le j \le i. \notag
\end{eqnarray}  
\end{lemma}

\begin{proof}
Let $x,y\ge 0$ be integers. 
We consider three separate cases:
\begin{flalign*}
\textbf{Case a}:& \quad y\ >\ \tfrac x 2 +\sqrt{x+\tfrac 9 4}+\tfrac 5 2,&\\
\textbf{Case b}:& \quad y\ =\ \tfrac x 2 +\sqrt{x+\tfrac 9 4}+\tfrac 3 2+r \text{ for some } r\in [0,1], and\\
\textbf{Case c}:& \quad y\ <\ \tfrac x 2 +\sqrt{x+\tfrac 9 4}+\tfrac 3 2.
\end{flalign*} 

\noindent
In \textbf{Case a} the first part of
condition \eqref{cond2} holds trivially. Since 
\[	y\ >\ \tfrac x 2 +\sqrt{x+\tfrac 9 4}+\tfrac 5 2
     \ >\ \tfrac x 2 +\big\lceil\sqrt{x+\tfrac 9 4}+\tfrac 1 2\big\rceil+ 1 ,
\]
condition \eqref{cond1} holds as well.
\smallskip

\noindent
In \textbf{Case c}
\[
	y\ <\   \tfrac x 2 +\sqrt{x+\tfrac 9 4}+\tfrac 3 2
	 \ \le\ \tfrac x 2 +\big\lceil\sqrt{x+\tfrac 9 4}+\tfrac 1 2\big\rceil+1,
\]
hence condition \eqref{cond1} is not satisfied.
On the other hand, observe that $r$ lies in the range from $0$ to~$1$ in the second part of condition \eqref{cond2}.
This shows us that condition \eqref{cond2} is not satisfied either.
\smallskip

\noindent
In \textbf{Case b} 
we prove the equivalence of condition \eqref{cond1} and \eqref{cond2} first for odd $x$, then for even~$x$.

Let $x$ be odd, $y = \tfrac x 2 +\sqrt{x+\tfrac 9 4}+\tfrac 3 2+r$ and $r\in[0,1]$.
Assume $x=2k+1$ for some $k \ge 0$.
We have
\begin{equation}
\sqrt{2k+\tfrac {13} 4}\ =\ y-k-r-2.    \label{r_equ1}
\end{equation}
Now
\begin{eqnarray*}
\tfrac x 2 +\big\lceil\sqrt{x+\tfrac 9 4}+\tfrac 1 2\big\rceil+1  & = & \\[2pt]
k + \big\lceil\sqrt{2k+\tfrac {13} 4}+\tfrac 1 2\big\rceil+\tfrac 3 2  
 &\overset{(\ref{r_equ1})}{=}&\\[-5pt]
k + \big\lceil y-k-r-\tfrac 3 2\big\rceil +\tfrac 3 2
 & = & \begin{cases}
    y+\tfrac 1 2& \text{if } r\in [0,\tfrac 1 2[,\\[1pt]
    y-\tfrac 1 2 & \text{if } r\in [\tfrac 1 2, 1].
\end{cases}
\end{eqnarray*}
This shows that condition \eqref{cond1} holds
if and only if $r\in [\tfrac 1 2, 1]$.

For $r\in [\tfrac 1 2,1]$ condition \eqref{cond2} is trivially satisfied. It remains to show that 
condition \eqref{cond2} does not hold for $r\in[0,\tfrac 1 2[$.
Assume by contradiction that condition \eqref{cond2} is satisfied for some  $x$ odd and $r\in[0,\tfrac 1 2[$.
The first part of condition \eqref{cond2} does not hold.
We will see that
\[
  r \ =\ i+\tfrac 1 2-\sqrt{(i+\tfrac 1 2)^2-2j}
\] 
with 
$i\ge 1$ and $0\le j \le i$
implies that $x$ is even.
Let 
\begin{eqnarray*}
y & = & \ \tfrac x 2+ 
   \sqrt{x+\tfrac {9}{4}}+\tfrac 3 2 +i+\tfrac 1 2-\sqrt{(i+\tfrac 1 2)^2-2j},\\
\noalign{\noindent then}
y-i-2 & = &\tfrac 1 2(x+\sqrt{4x+9}-\sqrt{4(i^2+i-2j)+1}). 
\end{eqnarray*}
So $\sqrt{4x+9}-\sqrt{4(i^2+i-2j)+1}$ is an integer of the same parity as $x$.

Either $4x+9=4(i^2+i-2j)+1$ or both $\sqrt{4x+9}$ and $\sqrt{4(i^2+i-2j)+1}$ are integers.
To see this, observe that
if $a,b\in \Z, c\in \Z_{\ne 0}$, then $\sqrt{a}-\sqrt{b}=c \Rightarrow \frac{a-b-c^2}{2c}=\sqrt{b}$.
This implies that $\sqrt{b}$ and hence $\sqrt{a}$ is a rational number, and since $a$ and $b$ are integers, 
$\sqrt{a}$ and $\sqrt{b}$ are integers as well.

In the first case, $x$ is even.  
In the second case, if $\sqrt{4x+9}$ and $\sqrt{4(i^2+i-2j)+1}$ are integers, then they are odd integers.
In both cases $x$ is an even integer, which contradicts the assumption.
Together, we obtain that conditions \eqref{cond1} and \eqref{cond2} are equivalent for odd $x$, for $r\in [0,1]$.

Let $x$ now be even, $x=2k$ for some $k \ge 0$, and
$y=\tfrac x 2+ \sqrt{x+\tfrac {9}{4}}+\tfrac 3 2+r$ for some $r\in[0,1].$ 
We have
\begin{eqnarray*}
y & = & \tfrac x 2+ \sqrt{x+\tfrac {9}{4}}+\tfrac 3 2+r\\
  & > & \tfrac x 2+ \big\lceil\sqrt{x+\tfrac {9}{4}}+\tfrac 1 2\big\rceil+r\\
  &\ge&\tfrac x 2+ \big\lceil\sqrt{x+\tfrac {9}{4}}+\tfrac 1 2\big\rceil.\\
\noalign{\noindent Since $x$ is even, this also shows that} 
y &\ge&\tfrac x 2+ \big\lceil\sqrt{x+\tfrac {9}{4}}+\tfrac 1 2\big\rceil +1,
\end{eqnarray*}
so condition \eqref{cond1} holds.
To see that condition \eqref{cond2} holds, we show if $x$ and $y$ 
are integers such that $x=2k$  for some $k \ge 0$ and
$y=\tfrac x 2+ \sqrt{x+\tfrac {9}{4}}+\tfrac 3 2+r$ for some $r\in[0,1]$,
then $r$ can be written as 
$r=i+\tfrac 1 2-\sqrt{(i+\tfrac 1 2)^2-2j}$ 
with integers $i\ge 1$, $0\le j \le i$.
For this, note that
\begin{eqnarray}
r & = & y-k-\tfrac 3 2 - \sqrt{2k+\tfrac {9}{4}}\label{r12}.\\
\noalign{\noindent Let $i\coloneqq y-k-2$ and
$j\coloneqq \tfrac{y(y-3)}{2}+\tfrac{k(k+1)}{2}-yk$. Then}
r & = & i+\tfrac 1 2-\sqrt{(i+\tfrac 1 2)^2-2j},\nonumber\\
y & = & \tfrac{i(i+3)}{2}-j+1, \label{def_y}\\
k & = & \ \tfrac{i(i+1)}{2}-j-1\label{def_k}.
\end{eqnarray}
Thus we have that $(y,k) \in \Z^2$ if and only if $(i,j) \in \Z^2$.
Observe that if $y,k\ge 0$  and $r\in [0,1]$, then by (\ref{r12}), $y\ge k+3$, so $i\ge1$.
It follows that $ r\in [0,1] $ if and only if $ 0\le j \le i$.
On the other hand, if $i\ge 1,j\ge 0$  and $r\in [0,1]$, then $j\le i$, so from (\ref{def_y}) it follows that $y\ge 0$ 
and from (\ref{def_k}) it follows that $k\ge 0$ if $i\ge 2$.
If $i=1$ and $j=0$, then $(y,k)=(3,0)$.
We exclude the special case $(i,j)=(1,1), r=1, (y,k)=(2,-1).$
This proves that $y$ and $k$ are non-negative integers with
\[
 r= y-k-\tfrac 3 2 - \sqrt{2k+\tfrac 9 4}\in [0,1]
\]
if and only if $i$ and $j$ are integers, $(i,j)\ne (1,1)$ with
\[
		i\ge 1,\ 
		0\le j\le i,\ 
		r=i+\tfrac 1 2-\sqrt{(i+\tfrac 1 2)^2-2j},
\]
so condition \eqref{cond2} is satisfied as well.
\end{proof}    

\begin{lemma}\label{gammacurves}
For $0\le r\le1$, let $\gamma_r$ be the algebraic curve $y=\tfrac x 2 +\sqrt{x+\tfrac {9} 4}+\tfrac 3 2+r$, 
restricted to $x\ge 0$.
To each curve $\gamma_{r_0}$ with $r_0\in [0,\tfrac 1 2]\cap\Q$,
there are two sequences of curves, $\gamma_{r_1(n)}$ and $\gamma_{r_2(n)},$ 
such that $|\gamma_{r_0}-\gamma_{r_1(n)}|$ and $|\gamma_{r_0}-\gamma_{r_2(n)}|$
converge to $0$.
Each $\gamma_{r_1(n)}$ contains an integer point $\big(x_1(n),y_1(n)\big)$  
from the set $A$ defined by \eqref{def:setA}
with $x_1(n)\rightarrow \infty$ as $n\rightarrow \infty$.
Each $\gamma_{r_2(n)}$ contains a point $\big(x_2(n),y_2(n)\big)$ 
from $\Z_{\ge 0}^2\backslash A$ 
with $x_2(n)\rightarrow \infty$ as $n\rightarrow \infty$.
\end{lemma}

\begin{proof}
Let $r_0 \in [0,\tfrac 1 2[$ be a rational number, $r_0=\tfrac p q, p,q \in \Z_{\ge0}$.
Let  $i=nq, j=np$ for some $n\in \Z_{\ge0},$
$ r_1(n)=i+\tfrac 1 2-\sqrt{(i+\tfrac 1 2)^2-2j}.$
Then $i,j \in \Z_{\ge0}, 0\le j< \tfrac i 2, r_1(n)\in[0,\tfrac 1 2[,
r_1(n)-r_0 \rightarrow 0 \text{ as } n \rightarrow \infty .$
If $n$ is an integer such that
 $n\ge \tfrac 1 q$, then $i\ge 1$
 and $0\le j \le i$.
Then, 
we have seen in Lemma~\ref{lemma_curves} that 
\[
	\big(x_1(n),y_1(n))\coloneqq (nq(nq+1)-2np-2,\tfrac{nq(nq+3)}{2}-np+1\big)
\]
 is an integer point with
$r_1(n)=y_1-\tfrac {x_1} 2 -\tfrac 3 2-\sqrt{x_1+\tfrac 9 4}\in [0,\tfrac 1 2[$ as $n\rightarrow \infty$. 
The point $(x_1(n),y_1(n))$ 
satisfies 
\[
	y_1(n)\ge \tfrac {x_1(n)} 2 +\big\lceil\sqrt{x_1(n)+\tfrac 9 4}+\tfrac 1 2\big\rceil+1,
\] 
which means it belongs to the set $A$ defined in \eqref{def:setA},
and $x_1(n)\rightarrow \infty$ as $n\rightarrow \infty$.

Now let
$i'=2nq, j'=2np$ for some  $n\in \Z_{\ge0},$
$r_2(n)=i'+1-\sqrt{(i')^2+2i'-2j'+\tfrac 5 4}$.
Then $i',j' \in \Z_{\ge0}, 0\le j'< \tfrac {i'} 2$ and
$r_2(n)-r_0 \rightarrow 0 \text{ as } n \rightarrow \infty.$
The point 
\[
	\big(x_2(n),y_2(n))\coloneqq (4n^2q^2+4nq-4np-1,2n^2q^2+4nq-2np+2\big)
\]
is an integer point with odd $x_2(n)$ and $r_2(n)=y_2-\tfrac {x_2} 2 -\tfrac 3 2-\sqrt{x_2+\tfrac 9 4}$
$\in [0,\tfrac 1 2[,$ where $x_2(n)\rightarrow \infty$ as $n\rightarrow \infty$.  
From Lemma~\ref{lemma_curves} it follows that
\[
	y_2(n)< \tfrac {x_2(n)} 2 +\big\lceil\sqrt{x_2(n)+\tfrac 9 4}+\tfrac 1 2\big\rceil+1,
\]
hence the point $(x_2(n),y_2(n))$ does not belong to the set $A$.
\end{proof}

\begin{theorem}\label{thm:setA}
The set
\[
	A\coloneqq \big\{(x,y) \in \Z_{\ge0}^2 : y\ge \tfrac x 2 +\big\lceil\sqrt{x+\tfrac 9 4}+
\tfrac 1 2\big\rceil+1\big\}
\]
is not an $\R$-semi-algebraic set of lattice points.
\end{theorem}

\begin{proof}
It follows from the proof of Lemma \ref{lemma_curves} that
$A$ can be written as the disjoint union of the sets $A_1$ and $A_2$, where
\begin{eqnarray*}
A_1 &\coloneqq& \Big\{(x,y) \in \Z_{\ge0}^2 : y\ge \tfrac x 2 +\sqrt{x+\tfrac 9 4}+2\Big\}\\
\noalign{\noindent and}
A_2 &\coloneqq& \Big\{(x,y) \in \Z_{\ge0}^2 : y= \tfrac x 2 +\sqrt{x+\tfrac 9 4}+\tfrac 3 2+r,
\quad r \in[0,\tfrac 1 2[,\\
&&\hphantom{\{(x,y) \in \Z_{\ge0}^2 : }\ r= i+\tfrac 1 2-\sqrt{(i+\tfrac 1 2)^2-2j}\\
&&\hphantom{\{(x,y) \in \Z_{\ge0}^2 : }\ 
\text{for some }i,j \in \Z_{\ge0},\ i\ge 1,\ 0\le j \le i\Big\}.
\end{eqnarray*}
The affine lattice spanned by $A$ is $\Lambda=\aff_{\Z}A=\aff_{\Z}A_1=\aff_{\Z}A_2=\Z^2$.
The set $A_1$ is the intersection of $\Z^2$ with the semi-algebraic set
\begin{eqnarray*}
S_1 &\coloneqq& \{(x,y)\in \R^2 : x,y \ge 0, y\le \tfrac x 2 +\sqrt{x+\tfrac 9 4}+2\}\\
    & =       & \{(x,y)\in \R^2 : x,y \ge 0, \tfrac{x^2}{4}+y^2-xy+x-4y+\tfrac 7 4\le 0\}.
\end{eqnarray*}
If there were a semi-algebraic set $S$ such that $A=S\cap \Z^2$, then
$S_2\coloneqq S\backslash S_1$ would be a semi-algebraic set and $A_2=S_2\cap \Z^2$. 
We will show that there is no such semi-algebraic set $S_2$ and hence no semi-algebraic set $S$.

Let $\gamma_r$ denote the curve $y= \tfrac x 2 +\sqrt{x+\tfrac 9 4}+\tfrac 3 2+r$.
For given $r$, $y\ge 0$ and $x\ge r-\tfrac 9 4$, 
$\gamma_r$
is a semi-algebraic set:
\[
	\gamma_r=\{(x,y)\in \R^2:\tfrac {x^2}{4}+y^2-xy+(r+\tfrac 1 2)x-(2r+3)y+r(r+3)=0\}.
\]
Set $A_2$ can now be written as
\begin{align*}
	A_2=
\big\{(x,y) \in \Z_{\ge0}^2\cap \gamma_r : r \in[0,\tfrac 1 2[, \, r= i+\tfrac 1 2-\sqrt{(i+\tfrac 1 2)^2-2j}&\\
	 \text{ for some }
i,j \in \Z_{\ge0}, i\ge 1, 1\le j \le i\big\}&.
\end{align*}
Consider an interval $[r,r+\varepsilon]\subset[0,\tfrac 1 2 ]$.
Take a rational $r_0\in \, ]r,r+\varepsilon[$.
By Lemma~\ref{gammacurves},
there exist infinitely many integer points both from $A_2$ and from $\R^2/A_2$ with arbitrarily high $x$-coefficient in the interval $[r,r+\varepsilon]$.
By Lemma~\ref{lemma:strip}, this implies that $A_2$ 
and hence $A$ cannot be the intersection of $\Z^2$ with any semi-algebraic set.
\end{proof}

Theorem~\ref{thm:setA} implies Theorem~\ref{thm:Pi4_12}: $\Pi^4_{12}$ is not a semi-algebraic set of lattice points.

\section{The set \texorpdfstring{$\boldsymbol{\Fset{d}{}}$}{Fd} for dimensions 6 and higher}\label{sec:dim6}

In this section we prove Theorem~\ref{thm:Fd}.
For this we need the notion of $g$-vectors of simplicial polytopes:
The \emph{$g$-vector} $g(P)\in \Z^{\lfloor\frac{d}{2}\rfloor+1}$ of a simplicial polytope is 
obtained from the $f$-vector $f(P)$ with
\[
	g_k= 
	\displaystyle \sum_{i=0}^k (-1)^{k-i}{{d-i+1}\choose{d-k+1}}f_{i-1} \
\text{ for } 0\ge k\ge \lfloor\frac{d}{2}\rfloor.
\]

\begin{definition}
Let $G^d$ denote the set of $g$-vectors of simplicial $d$-polytopes and 
let $G_{ij}^d$ denote the projection of the set of $g$-vectors of simplicial $d$-polytopes to the 
coordinates $i$ and $j$,
\[G_{ij}^d\coloneqq\big\{(g_i(P),g_j(P)) : P \text{ is a simplicial $d$-polytope}\big\}.\]
\end{definition}

\begin{lemma}\label{lemma:g2g3}
The set 
$G_{23}^d$
 is not an $\R$-semi-algebraic set of lattice points for $d\ge 6$.
\end{lemma}

\begin{proof}
The $g$-theorem (Billera \& Lee \cite{BiLe1} \cite{BiLe2} and Stanley \cite{Sta3})
gives us
\[
	G_{23}^d=\big\{(g_2(P),g_3(P))\in \Z^2 : g_2,g_3\ge 0, \partial^3(g_3)\le g_2\big\}.
\]
Here, 
\[
	\partial^3(g_3)\coloneqq{n_3-1 \choose 2}+\ldots+{n_i-1 \choose i}
\]	 
where $n_i,\ldots ,n_3$ are the unique integers such that
$1\le i \le n_i<\ldots<n_3$ and
\[
	g_3={n_3\choose 3}+\ldots+{n_i\choose i}.
\]
See Figure~\ref{plot:g2g3}.

\begin{figure}[htp]
  \centering
  \includegraphics[scale=.33]{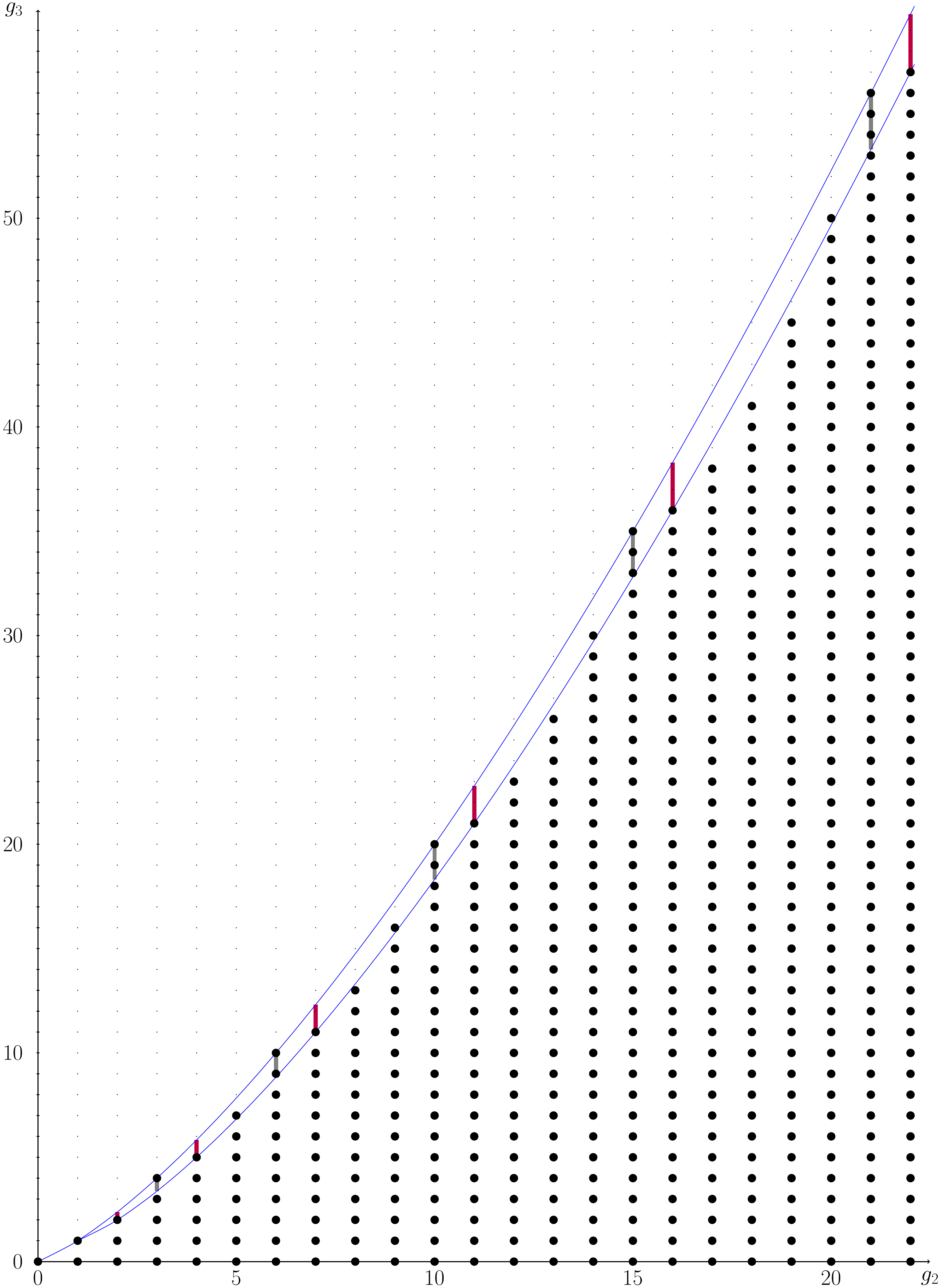}
\caption{The set $G_{23}^d$, $d\ge 6$}
\label{plot:g2g3}
\end{figure}

The affine lattice is $\Lambda=\Z^2.$
We show that this set is not an $\R$-semi-algebraic set of lattice points.
We will do this by considering the strip between 
the curve 
\[
	\gamma_0: g_3=\tfrac 1 2 g_2+\tfrac 1 3 g_2\sqrt{2g_2+\tfrac 1 4}
\]
through the points $({k \choose 2},{k+1 \choose 3})$ for $k\in \Z_{\ge0}$,
and the same curve, shifted by the vector $(1,1)\in \R^2$,
\[
	\gamma_1: g_3=\tfrac 1 2(g_2-1) +\tfrac 1 3 (g_2-1)\sqrt{2(g_2-1)+\tfrac 1 4}+1.
\]
We look at the points with $g_2={k \choose 2}$ and $g_2={k \choose 2} +1$ for any integer $k\ge 2$.
Observe
that points with $g_2={k \choose 2}$ in the strip satisfy
$\partial^3(g_3)\le g_2$ and
points with $g_2={k \choose 2}+1$ in the strip satisfy
$\partial^3(g_3)> g_2.$  
Additionally,
if $k\rightarrow \infty$, the number of points with $g_2={k \choose 2}$ 
and with $g_2={k \choose 2}+1$
in the strip goes to infinity.
By Lemma~\ref{lemma:strip}
this implies that the strip, and hence the whole set $G_{23}^d$, is not a semi-algebraic set of lattice points.
\end{proof}

Now we are ready  to prove Theorem~\ref{thm:Fd}:

\begin{proof}[Proof of Theorem~\ref{thm:Fd}]
The projection set $G_{23}^d$ is not semi-algebraic by Lemma~\ref{lemma:g2g3}.
This projection appears in the restriction of the set $G^d$ 
to $g_1\coloneqq g_2$ and $g_i\coloneqq 0$ for all $4\le i\le \lfloor\frac{d}{2}\rfloor$.
Therefore $G^d$ is not an $\R$-semi-algebraic set of lattice points
The transformation from the $g$-vector to the $f$-vector is unimodular.
Hence the set $\Fset{d}{s}$ of $f$-vectors of simplicial $d$-polytopes 
is not a semi-algebraic set of lattice points, for any $d\ge 6.$
The set $\Fset{d}{}$ of $f$-vectors of all $d$-polytopes 
is not a semi-algebraic set of lattice points, because its restriction to 
$2f_{d-2}=df_{d-1}$, 
the set of $f$-vectors of simplicial polytopes, 
is not a semi-algebraic set of lattice points.
\end{proof}

\subsection*{Acknowledgements}
Thanks to Isabella Novik and to an anonymous referee for insightful comments.

\newcommand{\SortNoop}[1]{}

\end{document}